\documentclass[11pt]{article}
\usepackage[T1]{fontenc}
\usepackage{times}
\usepackage{mathdots}
\usepackage{multirow}
\usepackage{textcomp}
\usepackage{amssymb,amscd,amsthm,latexsym}
\usepackage{amsmath}
\usepackage{graphicx}
\usepackage{float}
\usepackage{amssymb}
\usepackage{epstopdf}
\usepackage{verbatim}
\newcommand{\exampleqed}{\hfill \ensuremath{\square}}
\newtheorem{theorem}{Theorem}[section]

\theoremstyle{definition}

\newtheorem{algorithm}[theorem]{Algorithm}
\newtheorem{example}[theorem]{Example}

\newtheorem{rem}[theorem]{Remark}

\def\n{n \times n}

\def\l{\lambda}

\def\spacer#1{\vbox{\vspace*{#1}\hbox{\vphantom{M}}}}
\def\T{{\rm T}}
\title{\vskip10mm An application of Gelfand's formula in approximating the roots of polynomials
}
\author{Frank J. Hall\\ Department of Mathematics and Statistics\\ Georgia State University\\
	Atlanta, GA, USA\\ fhall@gsu.edu
		\, \\~~~\\
	Rachid M. Marsli
	\\ Preparatory Math Department\\
	King Fahd University of Petroleum and Minerals\\ Dhahran, KSA \\ rmarsliz@kfupm.edu.sa
	\, \\~~~\\
	Michael A. Stewart\\ Department of Mathematics and Statistics\\ Georgia State University\\
	Atlanta, GA, USA\\ mastewart@gsu.edu
}
\date{\today}
\begin{document}
	%
	%
	%
	\maketitle
	\begin{abstract}
	\noindent
	The purpose of this paper is to show how Gelfand's formula and balancing can be used to improve the upper and lower bounds of the spectrum of a companion matrix associated with a given real or complex polynomial. Examples and other related ideas are provided.
	\end{abstract}
	\noindent
	{\it AMS Subj. Class.:} 15A18; 15A42; 65F30
	\section{Introduction}

	An important classical topic is the approximation of the roots of polynomials with real or, more generally, complex coefficients. The connection between the eigenvalues of a matrix and the roots of a polynomial is a foundational concept in linear algebra and numerical analysis. For a given polynomial
	\[
	p(x) = a_nx^n + a_{n-1}x^{n-1} + \cdots + a_0,
	\]
	the companion matrix \(M\) is defined such that its eigenvalues correspond precisely to the roots of \(p(x)\). This correspondence arises from the characteristic polynomial of \(M\), which is identical to \(p(x)/a_n\), with $a_n$ being the leading coefficient of $p(x)$. The eigenvalues of \(M\) are exactly the roots of $p(x)$. This fundamental relationship serves as a cornerstone for numerous theoretical and computational studies.

	The problem of computing polynomial zeros is numerically challenging.  Zeros of higher multiplicity all reside in a single Jordan block in $M$, which means they are severely ill-conditioned.  Even a polynomial with a seemingly evenly distributed set of zeros is likely to have extremely ill-conditioned zeros, as with the Wilkinson polynomial $(x-1)(x-2)\cdots(x-20)$ described in \cite{wilk:63}.  Zeros of widely differing magnitudes and polynomial coefficients of widely differing magnitudes both lead to potential difficulties.  For these reasons, numerical computation of roots from a polynomial is often avoided when possible.  If the polynomial is the result of another computation, it is generally best to reformulate the problem to avoid computing a polynomial as an intermediate step.  A standard example is the solution of an eigenvalue problem using the $QR$ iteration applied to a matrix $A$ instead of by computing a characteristic polynomial from $A$.  Nevertheless, polynomials do arise directly in some applications and must sometimes be dealt with.  Improving the reliability and efficiency of methods for computing polynomial zeros is an important problem.

	Several numerical methods exploit the connection between the eigenvalues of \(M\) and the roots of \(p(x)\), using it either as a core principle or as a supporting tool for locating or refining the roots.  Forming a companion matrix and computing the eigenvalues using the $QR$ algorithm is perhaps the most common approach.  Backward stability in the usual matrix sense with errors on the companion matrix implies backward stability with errors placed on polynomial coefficients \cite{vdde:83, edmu:95}.  Hence the $QR$ algorithm is backward stable in a polynomial sense, although it does not avoid problems due to ill-conditioning.  The main disadvantage of the $QR$ algorithm is that it is costly, involving $O(n^3)$ operations.  More recent work on structured forms of the $QR$ iteration have reduced this to $O(n^2)$.  These algorithms are promising, but they have fairly large constant factors hidden in their $O(n^2)$ complexity.  While they appear to be backward stable in practice, we are not aware of a proof.

	Methods that work directly with the polynomials can be more efficient.  The \textit{Ehrlich-Aberth method} is an iterative root-finding algorithm that refines approximations for all roots of \(p(x)\) simultaneously \cite{AE}. While this method primarily operates directly on the polynomial, bounds on the eigenvalues of the companion matrix can inform initial guesses, significantly enhancing convergence, especially when the roots are well-separated. This hybrid use of eigenvalue bounds obtained from the companion matrix as starting points demonstrates the interplay between numerical linear algebra and classical polynomial root-finding techniques.
	
	We would also like to mention some other references related to our work. First, the Chapter 6 in the valuable book \cite{H} by P. Henrici deals with the classical problem of providing inclusion results for the roots of a polynomial. Next, by V. Pan et al, there are many important papers on the general	concept of "root finding/root radii", including \cite{Pan1} and \cite{Pan2}. Lastly, we mention the informative paper \cite{BF} by D. Bini, G. Fiorentino. In fact, we	shall use several of the examples from \cite{BF} to illustrate the techniques presented in this paper.

	A well-known eigenvalue localization technique is based on Gershgorin's theorem which provides regions in the complex plane (disks centered around diagonal entries of \(M\)) that are guaranteed to contain the eigenvalues of \(M\) \cite[Chapter 6]{HJ}. By associating the eigenvalues with the roots of \(p(x)\), these disks serve as initial bounds for the roots. This information can guide iterative techniques, such as Newton’s method or deflation, to narrow down and compute the roots with high precision. While not a standalone root-finding algorithm, Gershgorin-based refinement is widely used as a supporting strategy in hybrid approaches.


	Recent research has continued to explore the interplay between the eigenvalues of \(M\) and the roots of \(p(x)\), addressing both theoretical questions and numerical improvements.  Two fairly recent articles which make use of various companion matrices are the interesting and informative papers \cite{TDP} by De Ter\'an, Dopico, and P\'erez and \cite{V} by Vander Meulen and Vanderwoerd.  The reader can also see the references within those two papers, but we specifically mention the numerical work [1] by Bini and Fiorentino, and the paper \cite{AM} using Gershgorin discs by Melman.

	The purpose of this present paper is to show a different way to bound and locate the roots of polynomials.  For a given polynomial $p(x)$, we take powers of a companion matrix of $p(x)$ and also powers of the companion matrix of the reversal polynomial of $p(x)$. With the use of Gelfand’s formula, we obtain an annulus which contains the roots of $p(x)$. Matrix balancing is employed to realize better convergence results with relatively small matrix powers.  The basic principle is illustrated on some small examples and then the approach is tested on larger matrices with ill conditioned eigenvalues.  Some interesting results on the "spreads" of $p(x)$ are developed.  In the first five sections, we give all the foundational information of our methods in terms of the more classical companion matrix. In Section~6, we give some examples where other companion matrices such as Fiedler companion matrices are used. This opens the door for further research on this topic.

	\section{localization of the roots of a real polynomial}
	Let $p(x)$ be the monic real polynomial given by
	\begin{equation*}
		p(x) = x^n+a_{n-1}x^{n-1}+ \dots +a_0, \text{~~with~~} a_0 \neq 0 \text{~~and~~} n \geq 2.
	\end{equation*}
	Let $\l_1, \l_2, \dots, \l_n$ be the roots of $p(x)$ not necessarily distinct and listed in non-decreasing order of their absolute values, i.e.,
	$$|\l_1| \leq |\l_2| \leq \dots \leq |\l_n|.$$
	%
	A companion matrix associated with $p(x)$ is
	\begin{equation}\label{f17}
		C_1 \,=\, \left[\begin{array}{c c c c c}
			0      &  0     & ~\dots  & ~~~0      & -a_0      \\
			1      &  0     & ~\dots  & ~~~0      & -a_1      \\
			\ddots & \ddots & ~\ddots & ~~~\vdots & \vdots    \\
			0      & 0      & ~\ddots & ~~~0      & -a_{n-2}  \\
			0      & 0      & ~\ddots & ~~~1      & -a_{n-1}
		\end{array}\right].
	\end{equation}
This is the transpose matrix of what is known as Frobenius companion matrix. In fact, there are different forms of the companion matrix and the approach of this paper is valid if any one of these forms is alternatively used instead of $C_1$, with an expected difference in the bounds found for the roots of $p(x)$. To learn more about the different forms of companion matrices and also how they are used in bounding the roots of polynomials, we refer the reader to \cite{MF}, \cite{CG}, \cite{TDP}, and \cite{V}.  The roots of $p(x)$ are the eigenvalues of $C_1$ and the multiplicity of each root of $p(x)$ is exactly equal to its algebraic multiplicity as an eigenvalue of $C_1$.

	The equation $x^n + a_{n-1}x^{n-1} + \dots + a_1 x + a_0 = 0$ is equivalent to
	\[
		a_0 x^n\left(\frac{1}{a_0} + \frac{a_{n-1}}{a_0}\left(\frac{1}{x}\right) + \dots + \frac{a_1}{a_0} \left(\frac{1}{x}\right)^{n-1} +
		\left(\frac{1}{x}\right)^n \right)=0.
	\]
	It follows that $x_0$ is a root of $p(x)$ if and only if $1/x_0$ is a root of the monic polynomial
	\begin{equation}
		q(x) = x^n + \displaystyle{\frac{a_1}{a_0} x^{n-1} + \dots + \frac{a_{n-1}}{a_0} x + \frac{1}{a_0}}.
	\end{equation}
	This polynomial is called the monic reversal polynomial of $p(x)$, see \cite[p.366]{HJ}. Its eigenvalues are $1/\l_1, 1/\l_2, \dots, 1/\l_n$, with
	$$ \frac{1}{|\l_n|} \leq  \frac{1}{|\l_{n-1}|} \leq \dots \leq \frac{1}{|\l_1|}.$$
	It follows that a companion matrix associated with $q(x)$ is
	\begin{equation}\label{ff17}
		C_2 \,=\, \left[\begin{array}{c c c c c}
			0      &  0     & ~\dots  & ~~~0      & -1/a_0     \\
			1      &  0     & ~\dots  & ~~~0      & -a_{n-1}/a_1     \\
			\ddots & \ddots & ~\ddots & ~~~\vdots & \vdots    \\
			0      & 0      & ~\ddots & ~~~0      & -a_2/a_0 \\
			0      & 0      & ~\ddots & ~~~1      & -a_1/a_0
		\end{array}\right].
	\end{equation}

	For any square complex matrix $A$, let $\rho(A)$ and $N(A)$ be, respectively, the spectral radius and a submultiplicative matrix norm of $A$.  We assume without further comment that matrix norms submultiplicative.  It is well-known that
	\begin{equation}\label{ff20}
		\rho(A) \leq N(A).
	\end{equation}
	Applying this to $C_1$ and $C_2$, it follows that
	\begin{equation}\label{ff18}
		 N_2^{-1}(C_2) \leq |\l_1| \leq |\l_2| \leq \dots \leq |\l_n| \leq N_1(C_1),
	\end{equation}
	with $N_1$ and $N_2$ being matrix norms.  A geometric consequence of the above inequalities is that, in the complex plane, the roots of the polynomial $p(x)$ lie within the annulus $\varOmega_1$ given by
	\begin{equation} \label{k3}
		\varOmega_1 = \left\{ x + i y ~~ | ~~  N_2^{-1}(C_2) ~\leq ~\sqrt{x^2 + y^2} ~\leq ~ N_1(C_1) \right\}.
	\end{equation}

	Since we seek bounds that are easily computed from elements of the companion matrix, we focus on norms that have explicit expressions in terms of the matrix entries.  This includes induced one and infinity matrix norms as well as the Frobenius norm, which we denote by $\|A\|_F$.  If $A = [a_{ij}]$ is an $\n$ complex matrix then
	\begin{enumerate}
		\item $\|A\|_{\infty} = \displaystyle{\max_i \left\{\sum_{j=1}^n |a_{ij}|\right\}}$,
		\item $\|A\|_1 = \displaystyle{\max_j \left\{\sum_{i=1}^n |a_{ij}|\right\}}$,
		\item $\|A\|_F = \displaystyle{\sqrt{\sum_{i=1}^n \sum_{j=1}^n|a_{ij}|^2}}$.
	\end{enumerate}
		\begin{example}\label{ex1}
		Consider the polynomial
		$$p(x) = x^4 - x^3 -2x^2 + 6x -4.$$
		The roots of $p(x)$ are
		$$ \l_4 =-2,~ \l_3 = 1+i,~ \l_2 = 1-i \text{~~and~~} \l_1 = 1.$$
		The monic reversal polynomial of $p(x)$ is
		$$q(x) = x^4 - \frac{3}{2} x^3 + \frac{1}{2} x^2 + \frac{1}{4} x -\frac{1}{4}.$$
		The companion matrices associated, respectively, with $p(x)$ and $q(x)$ are
		$$
		C_1\,=\, \left[
		\begin{array}{c c c r }
			0 & 0 &   0 & ~4 \\
			1 & 0 &   0 & -6 \\
			0 & 1 &   0 & ~2 \\
			0 & 0 &   1 & ~1
		\end{array}
		\right]
		\text{~~ and ~~}
		C_2\,=\, \left[
		\begin{array}{c c  c r }
			0 & 0 &   0 & ~1/4 \\
			1 & 0 &   0 & -1/4 \\
			0 & 1 &   0 & -1/2 \\
			0 & 0 &   1 & ~3/2
		\end{array}
		\right].
		$$
		By simple calculations, we obtain
		\begin{center}
			\begin{tabular}{|c|l l|}\hline
				%
				\multirow{3}{3cm}{Upper bounds on $|\l_n| = |\l_4| = 2 $}
				&$\| C_1\|_{\infty} ~\approx$  	& $7$ 	  	\\ \cline{2-3}
				&$\| C_1\|_1 ~~\approx$       	& $13$    	\\ \cline{2-3}
				&$\| C_1\|_F ~\approx$       	 & $7.75$   \\ \hline
				\multirow{3}{3cm}{Lower bounds on $|\l_1| = 1$}
				& $\|C_2\|^{-1}_{\infty} ~\approx$   & $0.40$    \\ \cline{2-3}
				&$\|C_2\|^{-1}_1 ~~\approx $ 	    & $0.40$  	\\ \cline{2-3}
				&$\|C_2\|^{-1}_F ~\approx $ 	    & $0.42$  	\\ \hline
			\end{tabular}	\\
			\vskip 2mm
			Table~1
		\end{center}

	From this table we choose the best upper and lower bounds to construct an annulus within which lies all the roots of $p(x)$.
	See Figure 1.
	\begin{figure}
		\centering
		\includegraphics[scale=0.50]{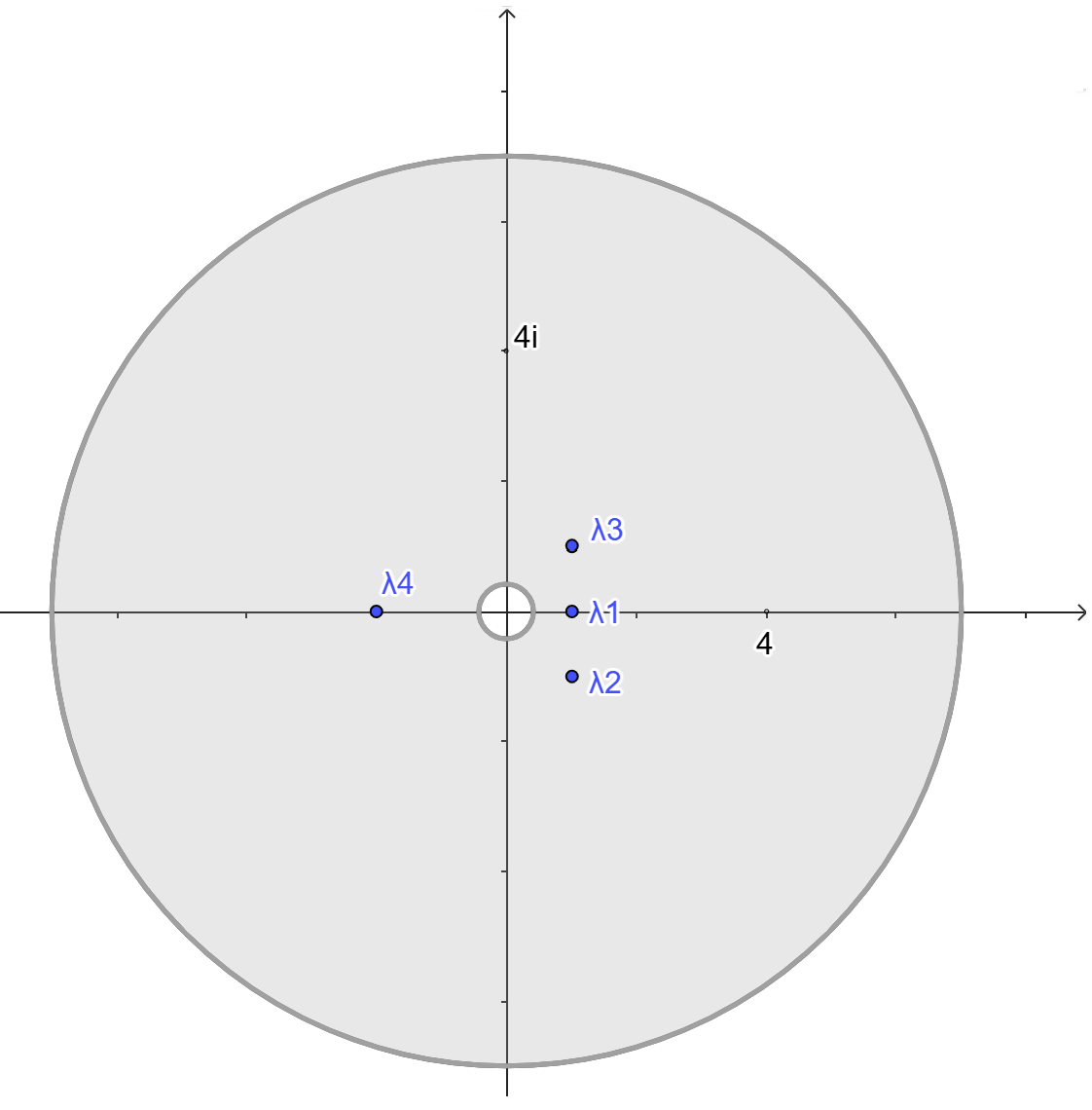}
		\caption{}
		All the roots of $p(x)$ lie within the annulus  bounded by the circles centered at the origin, with radii $0.42$ and $7$ selected from Table 1.
	\end{figure}
		\exampleqed
	\end{example}
	%
	In the above example, we have seen how the norms of the companion matrix are used to generate some upper and lower bounds for, respectively, the modulii of $\l_n$ and $\l_1$. We now show how the use of powers of the companion matrices facilitates finding sharper bounds.

	If $\l$ is an eigenvalue of $A$, then $\l^k$ is an eigenvalue of $A^k$ for every square complex matrix $A$ and every positive integer $k$.  Then it follows by (\ref{ff20}) that
	$$ |\l_1|^k \leq |\l_2|^k \leq \dots \leq |\l_n|^k  \leq N(C_1^k)$$
	and
	$$  \frac{1}{|\l_n|^k} \leq  \frac{1}{|\l_{n-1}|^k} \leq \dots \leq \frac{1}{|\l_1|^k} \leq N(C_2^k).$$
	Therefore, for all pairs $(k_1, k_2)$ in $\mathbb{N}^2$ and matrix norms $N_1$ and $N_2$,
	we have
	\begin{equation}\label{ff19}
		\frac{1}{\sqrt[k_2]{N_2\left(C_2^{k_2}\right)}} \leq |\l_1| \leq \dots \leq |\l_n| \leq \sqrt[k_1]{N_1\left(C_1^{k_1}\right)}.
	\end{equation}
	%
	%
	%
	%
	%
	%
	%

	%
	It follows from (\ref{ff19}) that the roots of the polynomial $p(x)$ lie within the annulus $\varOmega_{k_1,k_2}$ given by the following theorem.
	\begin{theorem}\label{t1}
		Using the same notation as above, the roots of the polynomial $p(x)$ are contained in the annulus defined by
		\begin{equation}\label{ff24}
			\varOmega_{k_1,k_2} = \left\{~~ x+iy ~~~~ \Bigg| ~~~~ \frac{1}{\sqrt[k_2]{N_2\left(C_2^{k_2}\right)}} ~ \leq ~ \sqrt{x^2 + y^2} ~ \leq ~ \sqrt[k_1]{N_1\left(C_1^{k_1}\right)} ~~\right\}.~~~~~~
		\end{equation}
	\end{theorem}

	Gelfand's formula \cite[Chapter 5]{HJ} states that for every $\n$ complex matrix $A$ with spectral radius $\rho(A)$ and for every matrix norm $N$,
	\begin{equation}\label{ff21}
		\rho(A) = \lim_{k\rightarrow \infty} \sqrt[k]{N(A^k)}~.
	\end{equation}

	The following is an immediate consequence of (\ref{ff21}).
	\begin{theorem}\label{t2}
		For any matrix norm $N$,
		\begin{equation}\label{ff22}
			\lim_{k \to \infty} N(C_1^{k})^{\frac {1}{k}} = |\l_n| \text{~~and~~}  \lim_{k \to \infty} N(C_2^{k})^{\frac {1}{k}} = \frac{1}{|\l_1|}.
		\end{equation}
	\end{theorem}
	\begin{rem}\label{r1}
		Clearly, Theorems \ref{t1} and \ref{t2} remain valid if $C_1$ is any nonsingular $n\times n$ complex matrix with $n\geq2$, and $C_2$ is its inverse matrix or any matrix for which the smallest and largest eigenvalues in absolute value are, respectively, $1/\l_n$ and $1/\l_1$.
	\end{rem}

	It may happen in some cases that the size of the inclusion region (\ref{k3}) obtained by using $C_1$ and $C_2$ is unreasonably large compared to the distribution of the roots of $p(x)$. A good alternative is to produce inclusion regions of the form (\ref{ff24}) by considering the integer powers of $C_1$ and $C_2$.  This is ensured by (\ref{ff22}) which means that we can obtain bounds that are arbitrary close to $|\l_n|$ and $|\l_1|$ for sufficiently large values of $k_1$ and $k_2$.  An algorithm based on Theorem \ref{t1} to compute these types of localization regions consists of the following steps.
	\begin{algorithm}\label{a1}~~~
	\begin{enumerate}
	\item Input the coefficients of the real or complex monic polynomial $p(x)$ with a nonzero constant coefficient.
		\item Compute the companion matrix $C_1$ of $p(x)$.
		\item Compute $C_1^{k_{1}}$, the matrix $C_1$ raised to the integer power $k_1$.
		\item Compute $\|C_1^{k_1}\|$, a matrix norm of $C_1^{k_1}$. (Use the one, infinity or Frobenius norm).
		\item $b_1 = \sqrt[k_1]{\|C_1^{k_1}\|}$ is an upper bound of $|\l_n|$ the largest root in absolute value of $p(x)$.
		\item Compute the coefficients of $q(x)$, the monic reversal polynomial of $p(x)$.
		\item Compute the companion matrix $C_2$ of $q(x)$.
		\item Compute $C_2^{k_{2}}$, the matrix $C_2$ raised to the integer power $k_2$.
		\item Compute $\|C_2^{k_2}\|$, a matrix norm of $C_2^{k_2}$. (Use the one, infinity or Frobenius norm).
		\item $b_2 = \sqrt[-k_2]{\|C_2^{k_2}\|}$ is a lower bound of $|\l_1|$, the smallest root in absolute value of $p(x)$.
		\item The circles centered at the origin with radii $b_1$ and $b_2$ are the boundaries of an annulus that contains all the roots of $p(x)$.
	\end{enumerate}
	\end{algorithm}
	\begin{example}\label{ex2}
		Consider the polynomial
		$$p(x) = x^8 +8x^7+14x^6-28x^5-81x^4-8x^3-14x^2+28x+80.$$
		The roots of this polynomial are
		$$\l_8 =4, l_7 = -3+i, l_6=-3-i, l_5 = 2, l_4 = -1, l_3=i, l_2=-i \text{~~and~~} l_1=1.$$
		The monic reversal polynomial of $p(x)$ is
		$$q(x) = x^8 +\frac{7}{20} x^7 -\frac{7}{40} x^6-\frac{1}{10} x^5-\frac{81}{80} x^4-\frac{7}{20} x^3+\frac{7}{40} x^2+\frac{1}{10} x+\frac{1}{80} .$$

		The companion matrices associated with, respectively, with $p(x)$ and $q(x)$ are given by
		{\fontsize{8}{11}\selectfont
		$$
		C_1\,=\, \left[
		\begin{array}{c c c c c c c r }
			0 & 0 &   0 &  0  & 0 &   0 &  0  & -80\\
			1 & 0 &   0 &  0  & 0 &   0 &  0  & -28\\
			0 & 1 &   0 &  0  & 0 &   0 &  0  &  14\\
			0 & 0 &   1 &  0  & 0 &   0 &  0  &   8\\
			0 & 0 &   0 &  1 &  0 &   0 &  0 &   81\\
			0 & 0 &   0 &  0 &  1 &   0 &  0 &   28\\
			0 & 0 &   0 &  0 &  0 &   1 &  0 &  -14\\
			0 & 0 &   0 &  0 &  0 &   0 &  1 &   -8
		\end{array}
		\right]
		\text{~~ and ~~}
		C_2\,=\, \left[
		\begin{array}{c c c c c c c r }
			0 & 0 &   0 &  0  & 0 &   0 &  0  &   -1/80  \\
			1 & 0 &   0 &  0  & 0 &   0 &  0  &   -1/10  \\
			0 & 1 &   0 &  0  & 0 &   0 &  0  &   -7/40  \\
			0 & 0 &   1 &  0  & 0 &   0 &  0  &    7/20  \\
			0 & 0 &   0 &  1 &  0 &   0 &  0 &     81/80 \\
			0 & 0 &   0 &  0 &  1 &   0 &  0 &     1/10  \\
			0 & 0 &   0 &  0 &  0 &   1 &  0 &     7/40  \\
			0 & 0 &   0 &  0 &  0 &   0 &  1 &    -7/20
		\end{array}
		\right]
		$$
		}
		Calculations of some powers of $C_1$ and $C_2$ along with their one, infinity and Frobenius norms lead to the following table.
		{\fontsize{8}{11}\selectfont
			\begin{center}
				\begin{tabular}{|c|c|c|c|c|c|c|c|}\hline
					&$k$										 & $1$ &  $2$     &  $8$    & $16$    & $32$     & $64$      \\ \hline
					\multirow{3}{3cm}{Upper bounds on $|\l_n| = |\l_8| = 4 $}
					&$\sqrt[k]{\| C_1^k\|_{\infty}}\approx$  & $82$ 	&  $26.87$     &  $7.51$ & $5.45$ & $4.67$   & $4.32$  \\ \cline{2-8}
					&$\sqrt[k]{\| C_1^k\|_1}\approx$       	 & $261$	&  $44.12$     &  $8.22$ & $5.70$ & $4.78$   & $4.37$  \\ \cline{2-8}
					&$\sqrt[k]{\| C_1^k\|_F}\approx$       	 & $122.70$ &  $30.97$     &  $7.63$ & $5.49$ & $4.69$   & $4.33$  \\ \hline
					\multirow{3}{3cm}{Lower bounds on $|\l_1| = 1$} 
					& $\sqrt[-k]{\|C_2^k\|_{\infty}} \approx$   & $0.50$ &  $0.69$     &  $0.85$ & $0.92$ & $0.96$   & $0.98$  \\ \cline{2-8}
					&$\sqrt[-k]{\| C_2^k\|_1} \approx $ 	    & $0.44$ &  $0.66$     &  $0.90$ & $0.95$ & $0.97$   & $0.99$  \\ \cline{2-8}
					&$\sqrt[-k]{\| C_2^k\|_F} \approx $ 	    & $0.35$ &  $0.59$     &  $0.87$ & $0.93$ & $0.97$   & $0.98$  \\ \hline
				\end{tabular}	\\
				\vskip 2mm
				Table~2
			\end{center}
		}
		\begin{figure}
			\centering
			\includegraphics[scale=0.50]{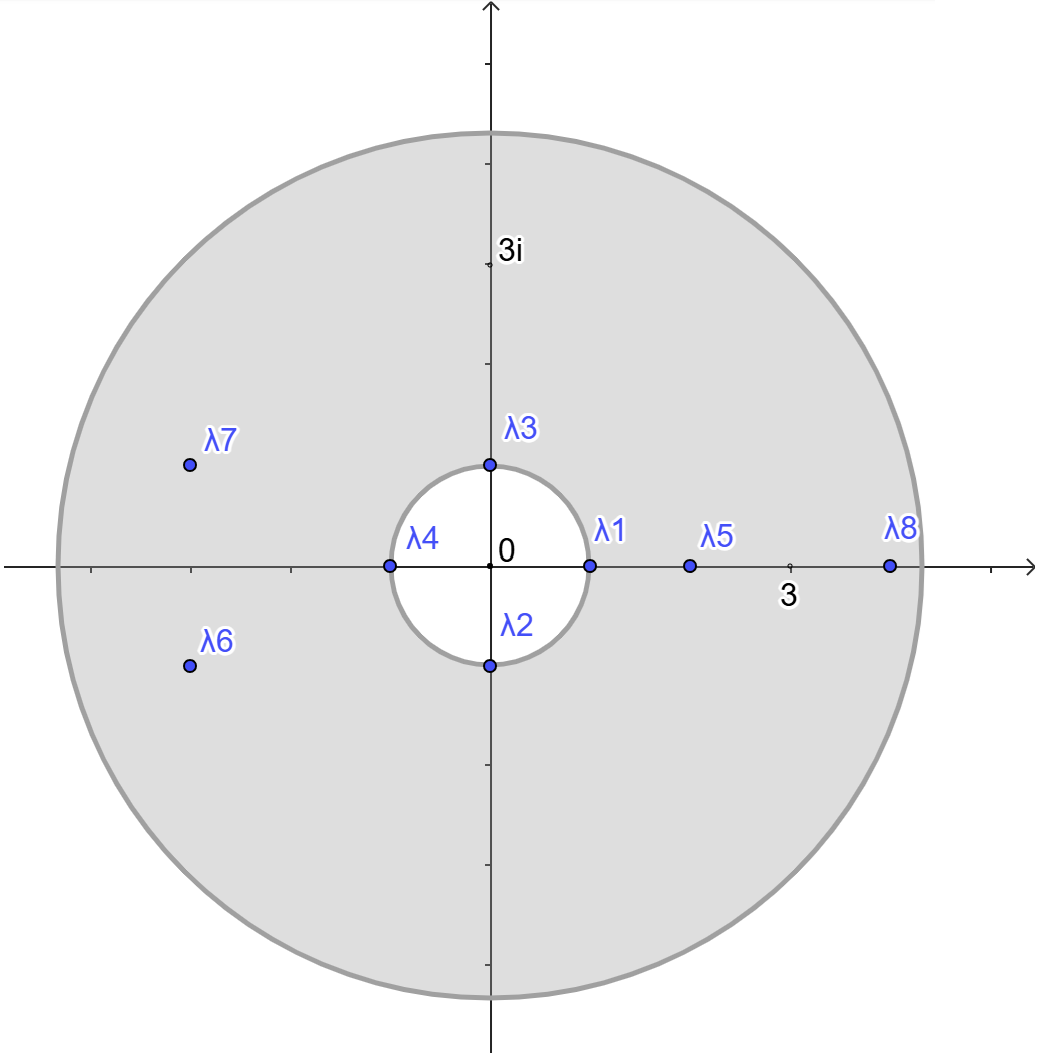}
			\caption{}
			All the roots of $p(x)$ lie within the annulus  bounded by the two circles centered at the origin, with radii $0.99$ and $4.32$ selected from Table 2.
		\end{figure}
	\exampleqed
	\end{example}

	\section{Refining the roots location through matrix balancing of $C_1$ and $C_2$}

	Balancing a square real matrix via diagonal similarity transformations is a numerical technique used to decrease errors in eigenvalue computations. This process aims to reduce the norm of the matrix while preserving its eigenvalues, which can lead to more accurate numerical results.  The process of balancing an $n \times n$ irreducible complex matrix $A_1=[a_{ij}]$ can be done as follows (see \cite{JG}).  Compute the matrix $A_2 = D_1 A_1 D_1^{-1}$, where $D_1 = [d_{ij}]$ is an $n\times n$ diagonal matrix with diagonal entries defined as follows:
	$$
		d_{11} = \sqrt{\frac{ \sum_{s=2}^{n}|a_{s1}|}{\sum_{s=2}^{n}|a_{1s}|}},
	$$
	$$
	d_{ii} = \sqrt{\frac{\sum_{s=1}^{i-1}|a_{si}|d_{ss} + \sum_{s=i+1}^{n}|a_{si}|}{\sum_{s=1}^{i-1}|a_{is}|/d_{ss} + \sum_{s=i+1}^{n}|a_{is}|}},
	$$
	for $i=2, 3, \dots, n-1$, and
	$$
	d_{nn} = \sqrt{\frac{\sum_{s=1}^{n-1}|a_{sn}|d_{ss} }{\sum_{s=1}^{n-1}|a_{ns}|/d_{ss}}}.
	$$
	Apply the same process to $A_2$ to get matrix $A_3 = D_2 A_2 D_2^{-1}$ and continue iterating this way for the required number of iterations.  The obtained matrix will typically have a smaller norm than the original matrix $A_1$.  In general, this technique is known to work well with companion matrices which are sparse and irreducible and it can improve the norms noticeably after few iterations only.  Further, if sparsity is exploited, each iteration of the balancing algorithm requires only $O(n)$ operations.  Small powers of a companion matrix also have sparsity structure, so a balancing step applied to $C_1^k$ is also an $O(n)$ computation if $k$ is bounded independently of $n$.

	By incorporating the technique of matrix balancing into Algorithm \ref{a1}, we obtain the following updated algorithm which summarizes the main steps needed to derive bounds and locations of the polynomial roots.
	\begin{algorithm}\label{a2}~~~
		\begin{enumerate}
			\item Input the coefficients of the real or complex monic polynomial $p(x)$ with a nonzero constant coefficient.
			\item Compute the companion matrix $C_1$ of $p(x)$.
			\item Compute $C_1^{k_1}$, the matrix $C_1$ raised to the integer power $k_1$.
			\item Balance $C_1^{k_1}$; by starting with a few iterations, this could be enough.
			\item Compute $\|C_1^{k_1}\|$, a matrix norm of $C_1^{k_1}$. (Use one, infinity or Frobenius norm).
			\item $b_1 = \sqrt[k_1]{\|C_1^{k_1}\|}$ is an upper bound of $|\l_n|$ the largest root in absolute value of $p(x)$.
			\item Compute the coefficient of $q(x)$, the monic reversal polynomial of $p(x)$.
			\item Compute the companion matrix $C_2$ of $q(x)$.
			\item Compute $C_2^{k_2}$, the matrix $C_2$ raised to the integer power $k_2$.
			\item Balance $C_2^{k_2}$; using a sufficient number of iterations.
			\item Compute $\|C_2^{k_2}\|$, a matrix norm of $C_2^{k_2}$. (Use one, infinity or Frobenius norm).
			\item $b_2 = \sqrt[-k_2]{\|C_2^{k_2}\|}$ is a lower bound of $|\l_1|$, the smallest root in absolute value of $p(x)$.
			\item The circles centered at the origin with radii $b_1$ and $b_2$ are the boundaries of an annulus that contains all the roots of $p(x)$.
		\end{enumerate}
	\end{algorithm}
			~
	\begin{example}\label{ex3}
		We reconsider the polynomial $p(x)$ in Example \ref{ex3} and we apply the technique of matrix balancing with 3 iterations to the same powers of $C_1$ and $C_2$ provided in Table~2.  The obtained bounds are given in the next table where $F_k$ is the matrix obtained by balancing $C_1^k$ and $G_k$ is the matrix obtained by balancing $C_2^k$.  {\fontsize{8}{11}\selectfont
			\begin{center}
				\begin{tabular}{|c|c|c|c|c|c|c|c|}\hline
					&$k$										 & $1$ &  $2$     &  $8$    & $16$    & $32$     & $64$      \\ \hline
					\multirow{3}{3cm}{Upper bounds on $|\l_n| = |\l_8| = 4 $}
					&$\sqrt[k]{\| F_k\|_{\infty}}\approx$  & $14.74$ 	&  $10.15$     &  $5.74$ & $4.77$ & $4.37$   & $4.18$  \\ \cline{2-8}
					&$\sqrt[k]{\| F_k\|_1}\approx$       	 & $14.74$	&  $10.15$     &  $5.74$ & $4.77$ & $4.37$   & $4.18$  \\ \cline{2-8}
					&$\sqrt[k]{\| F_k\|_F}\approx$       	 & $12.11$  &  $8.73$      &  $5.57$ & $4.70$ & $4.33$   & $4.16$  \\ \hline
					\multirow{3}{3cm}{Lower bounds on $|\l_1| = 1$}
					& $\sqrt[-k]{\|G_k\|_{\infty}} \approx$   & $0.56$ &  $0.76$     &  $0.95$ & $0.97$ & $0.987$   & $0.993$  \\ \cline{2-8}
					&$\sqrt[-k]{\| G_k\|_1} \approx $ 	    & $0.56$ &  $0.76$     &  $0.95$ & $0.97$ & $0.987$   & $0.993$  \\ \cline{2-8}
					&$\sqrt[-k]{\| G_k\|_F} \approx $ 	    & $0.38$ &  $0.63$     &  $0.91$ & $0.95$ & $0.976$   & $0.988$  \\ \hline
				\end{tabular} \\
				\vskip 2mm
				Table~3
			\end{center}
		}

		For this particular polynomial, we have a few observations.
		\begin{enumerate}
		\item By comparing Table~3 with Table~2, we can see that the bounds obtained with the use of matrix balancing are better than those obtained without it for each of the chosen norms and powers.
		\item The bounds obtained by using the one and infinity norms are close to each other but slightly different from those obtained by using Frobenius norm
		\item The bounds on $|\l_n|$ obtained with Frobenius norm are slightly better than those obtained with one and infinity norms.
		\item The bounds on $|\l_1|$ obtained with one and infinity norms are slightly better than those obtained with Frobenius norm.
		\end{enumerate}
		\exampleqed
	\end{example}
	%

\section{Numerical Results for Larger Matrices}
\label{sec:numer-results-larg}

	Applying Gelfand's theorem to larger matrices by computing higher powers of $C_1$ or $C_2$ is likely to be inefficient.  In general, the matrix $C_1^k$ has the form
	\begin{equation*}
		C_1^k =
		\begin{bmatrix}
			0_{k\times (n-k)} & C_{12}\\
			I_{n-k} & C_{22}
		\end{bmatrix}
	\end{equation*}
	so that only the last $k$ columns of $C_1^k$ need to be computed to determine $N_1(C_1^k)$.  If $k$ is bounded and small relative to $n$, then computing $N_1(C_1^k)$ requires $O(n)$ operations.  Balancing $C_1^{k}$ then also requires only $O(n)$ operations.  If, instead, we have no bound on $k$ and the computation of $C_1^k$ the powers becomes $O(n^3)$ or worse, then it may require as many operations to compute the matrix powers as it would be to compute the eigenvalues of $C_1$ using an unstructured $QR$ iteration.  If we seek inclusion regions that are fast to compute and potentially useful for the iterative computation of polynomial zeros, we would like to be able to obtain useful inclusion regions from small powers of large matrices.  Similar comments apply to $C_2^k$.

	In this section we consider mostly small powers of some difficult families of companion matrices taken from \cite{BF}.  We specifically consider the following
	\begin{enumerate}
		\item The Laguerre polynomial $L_{100}(x)$ where
			\begin{equation*}
				L_0(x) = 1, \qquad L_1(x) = 1-x, \qquad
				L_{i+1}(x) = (2i+1-x) L_i(x) -i^2p_{i-1}(x).
			\end{equation*}
			Note that while these zeros are very ill-conditioned as functions of polynomial coefficients, they can be obtained stable way as the eigenvalues of a symmetric Jacobi matrix, which is how we have computed them in the tests.
		\item The characteristic polynomial of a $100\times 100$ 9-diagonal complex Toeplitz matrix with nonzero diagonals in positions $(-2, -1, 0, 1, 2, 3, 4, 5, 6)$ (that is the second subdiagonal through the sixth superdiagonal) and values on those diagonals given by $(-i, -3, 0, 10, 1, i, 28, -3, 1)$.
		\item The $50$ degree polynomial $(x-1)^{50}$.  The companion matrix then has a single extremely ill-conditioned eigenvalue $\lambda = 1$ of high multiplicity in a single Jordan block.
		\item A polynomial (called LSR1 in \cite{BF})
			\begin{equation*}
				(x^{50} + 1)(x^2+ax+a^{-1})
			\end{equation*}
			for $a=10^{20}$.  This polynomial has a very small zero near $-1/a^2 = -10^{-20}$ and a very large zero near $-a=-10^{20}$ and provides some insight into the process behaves when eigenvalues have a very large range of moduli.
	\end{enumerate}

	All of these have ill-conditioned zeros.  The polynomials and companion matrices were generated in extended precision of 256 bits using the \verb|BigFloat| type in Julia.  The computations of the spectral radii to which we compare the bounds were done in extended precision.  Prior to taking power and balancing, the matrices were converted to double precision, although we did run tests with higher precision to make sure that the results were not compromised significantly by the use of double precision.  One limitation of double precision is a greater possibility of overflow.  The routine for computing powers used scaling to avoid overflow.

	As noted, on larger problems, forming high powers of companion matrices will be inefficient since it results in matrices with dense blocks.  For this reason, we have focused mostly on lower powers, although we have included a single result in our tables for the power $k=100$.  Since balancing is very efficient on small powers of companion matrices, we have used however many balancing steps seemed to give the best results.  In all of the test problems, 20 balancing iterations seemed to give good bounds, with minimal benefit from additional iterations.  We used 3 iterations in one experiment with the Laguerre polynomial, just to show how imperfect balancing interacts with the improvement in the bounds from taking powers.  Higher powers have more benefit on matrices that are not as fully balanced.

	Table~\ref{tb:laguerre3} and Table~\ref{tb:laguerre20} show the bounds for $L_{100}(x)$ with 3 and 20 balancing iterations respectively.  Even 3 iterations reduces the upper bound and increases the lower bound dramatically.  In this case, there is significant improvement from taking powers, although even for $k=100$, the upper bound is not particularly close to $\rho(C_1)$.  In contrast, balancing more carefully with 20 iterations yields a big improvement for $k=1$, but balancing higher powers improve the bounds more slowly.  The largest improvement is seen in going to $k=2$.  It might seem disappointing that increasing $k$ beyond a fairly low power gives only incremental improvement, but this can be viewed positively: Balancing is fast, higher powers are slow, and the greatest benefit is seen with the fastest operations.  All later results given here are for $20$ balancing iterations, but we have looked at a smaller number of iterations in all the examples.  A similar pattern holds for the others.

	\setcounter{table}{3} 
	\begin{table}
		\begin{center}
			\begin{tabular}{|r|r|r|r|r|r|}
				\hline
				\multicolumn{1}{|c|}{} & \multicolumn{2}{|c|}{$\rho(C_1)=374.984$}
					& \multicolumn{2}{|c|}{$1/\rho(C_2)=0.0143861$} \\\hline
				\spacer{10pt}
				$k$ & $\sqrt[k]{\|C_1^k\|_1}$ & $\sqrt[k]{\|F_k\|_1}$
					& $\sqrt[-k]{\|C_2^k\|_1}$ & $\sqrt[-k]{\|G_k\|_1}$ \\[6pt]\hline
				\spacer{4pt}
				1 & $2.609 \times 10^{165}$ & $2.057 \times 10^{17}$ & $3.578 \times 10^{-8}$ & $0.005732$ \\
				2 & $5.107 \times 10^{84}$ & $3.485 \times 10^{9}$ & $1.901 \times 10^{-5}$ & $0.008391$ \\
				3 & $5.105 \times 10^{57}$ & $8.301 \times 10^{6}$ & $0.0001689$ & $0.009814$ \\
				4 & $1.465 \times 10^{44}$ & $4.03 \times 10^{5}$ & $0.000511$ & $0.01074$ \\
				5 & $1.041 \times 10^{36}$ & $6.73 \times 10^{4}$ & $0.0009956$ & $0.01137$ \\
				10 & $4.128 \times 10^{19}$ & $4426$ & $0.003784$ & $0.01279$ \\
				100 & $2.768 \times 10^{4}$ & $1190$ & $0.01259$ & $0.01422$ \\\hline
			\end{tabular}
		\end{center}
		\caption{Degree 100 Laguerre Polynomial, 3 iterations}
		\label{tb:laguerre3}
	\end{table}

	\begin{table}
		\begin{center}
			\begin{tabular}{|r|r|r|r|r|r|}
				\hline
				\multicolumn{1}{|c|}{} & \multicolumn{2}{|c|}{$\rho(C_1)=374.984$}
				& \multicolumn{2}{|c|}{$1/\rho(C_2)=0.0143861$} \\\hline
				\spacer{10pt}
				$k$ & $\sqrt[k]{\|C_1^k\|_1}$ & $\sqrt[k]{\|F_k\|_1}$ 
				& $\sqrt[-k]{\|C_2^k\|_1}$ & $\sqrt[-k]{\|G_k\|_1}$ \\[6pt]\hline
				\spacer{4pt}
				1 & $2.609 \times 10^{165}$ & $2.267 \times 10^{4}$ & $3.578 \times 10^{-8}$ & $0.006029$ \\
				2 & $5.107 \times 10^{84}$ & $1.274 \times 10^{4}$ & $1.901 \times 10^{-5}$ & $0.008392$ \\
				3 & $5.105 \times 10^{57}$ & $9931$ & $0.0001689$ & $0.009815$ \\
				4 & $1.465 \times 10^{44}$ & $8231$ & $0.000511$ & $0.01074$ \\
				10 & $4.128 \times 10^{19}$ & $4273$ & $0.003784$ & $0.01279$ \\
				100 & $2.768 \times 10^{4}$ & $769.4$ & $0.01259$ & $0.01422$ \\\hline
			\end{tabular}
		\end{center}
		\caption{Degree 100 Laguerre Polynomial, 20 iterations}
		\label{tb:laguerre20}
	\end{table}

	Table~\ref{tb:toeplitz} shows the characteristic polynomial of the Toeplitz matrix, which has complex eigenvalues with results similar to the Laguerre polynomial.  The same is true of the polynomial with a zero of high multiplicity in Table~\ref{tb:ones}.  Finally, we note that LSR1 did not seem to pose much of a challenge.  The norm of the companion matrices accurately approximates the spectral radii after balancing.  The results were ideal or nearly ideal even without balancing.

	\begin{table}
		\begin{center}
			\begin{tabular}{|r|r|r|r|r|r|}
				\hline
				\multicolumn{1}{|c|}{} & \multicolumn{2}{|c|}{$\rho(C_1)=17.0562$}
					& \multicolumn{2}{|c|}{$1/\rho(C_2)=0.4202$} \\\hline
				\spacer{10pt}
				$k$ & $\sqrt[k]{\|C_1^k\|_1}$ & $\sqrt[k]{\|F_k\|_1}$
					& $\sqrt[-k]{\|C_2^k\|_1}$ & $\sqrt[-k]{\|G_k\|_1}$ \\[6pt]\hline
				\spacer{4pt}
				1 & $2.454 \times 10^{81}$ & $243.7$ & $0.0008301$ & $0.06895$ \\
				2 & $4.954 \times 10^{40}$ & $81.2$ & $0.01189$ & $0.1137$ \\
				3 & $1.939 \times 10^{28}$ & $69.78$ & $0.0364$ & $0.1455$ \\
				4 & $2.29 \times 10^{21}$ & $63.29$ & $0.07098$ & $0.1774$ \\
				5 & $4.07 \times 10^{17}$ & $58.95$ & $0.1113$ & $0.2038$ \\
				10 & $3.26 \times 10^{9}$ & $47.49$ & $0.2104$ & $0.2885$ \\
				100 & $121.2$ & $20.86$ & $0.3919$ & $0.4057$ \\\hline
			\end{tabular}
		\end{center}
		\caption{Degree 100 Toeplitz Characteristic Polynomial, 20 iterations}
		\label{tb:toeplitz}
	\end{table}

	\begin{table}
		\begin{center}
			\begin{tabular}{|r|r|r|r|r|r|}
				\hline
				\multicolumn{1}{|c|}{} & \multicolumn{2}{|c|}{$\rho(C_1)=1.00$}
					& \multicolumn{2}{|c|}{$1/\rho(C_2)=1.00$} \\\hline
				\spacer{10pt}
				$k$ & $\sqrt[k]{\|C_1^k\|_1}$ & $\sqrt[k]{\|F_k\|_1}$
					& $\sqrt[-k]{\|C_2^k\|_1}$ & $\sqrt[-k]{\|G_k\|_1}$ \\[6pt]\hline
				\spacer{4pt}
				1 & $1.126 \times 10^{15}$ & $105$ & $8.882 \times 10^{-16}$ & $0.009524$ \\
				2 & $2.349 \times 10^{8}$ & $63.79$ & $4.257 \times 10^{-9}$ & $0.01568$ \\
				3 & $1.113 \times 10^{6}$ & $49.65$ & $8.981 \times 10^{-7}$ & $0.02014$ \\
				4 & $6.963 \times 10^{4}$ & $41.15$ & $1.436 \times 10^{-5}$ & $0.0243$ \\
				5 & $1.251 \times 10^{4}$ & $35.65$ & $7.995 \times 10^{-5}$ & $0.02805$ \\
				10 & $317.4$ & $21.35$ & $0.003151$ & $0.04683$ \\
				100 & $4.356$ & $3.755$ & $0.2296$ & $0.2663$ \\\hline
			\end{tabular}
		\end{center}
		\caption{$(x-1)^{50}$, 20 iterations}
		\label{tb:ones}
	\end{table}

	\begin{table}
		\begin{center}
			\begin{tabular}{|r|r|r|r|r|r|}
				\hline
				\multicolumn{1}{|c|}{} & \multicolumn{2}{|c|}{$\rho(C_1)=1.00\times 10^{20}$}
				& \multicolumn{2}{|c|}{$1/\rho(C_2)=1.00\times 10^{-40}$} \\\hline
				\spacer{10pt}
				$k$ & $\sqrt[k]{\|C_1^k\|_1}$ & $\sqrt[k]{\|F_k\|_1}$
						& $\sqrt[-k]{\|C_2^k\|_1}$ & $\sqrt[-k]{\|G_k\|_1}$ \\[6pt]\hline
				\spacer{4pt}
				1 & $2 \times 10^{20}$ & $1 \times 10^{20}$ & $1 \times 10^{-40}$ & $1 \times 10^{-40}$ \\
				2 & $1.414 \times 10^{20}$ & $1 \times 10^{20}$ & $1 \times 10^{-40}$ & $1 \times 10^{-40}$ \\
				3 & $1.26 \times 10^{20}$ & $1 \times 10^{20}$ & $1 \times 10^{-40}$ & $1 \times 10^{-40}$ \\
				4 & $1.189 \times 10^{20}$ & $1 \times 10^{20}$ & $1 \times 10^{-40}$ & $1 \times 10^{-40}$ \\
				5 & $1.149 \times 10^{20}$ & $1 \times 10^{20}$ & $1 \times 10^{-40}$ & $1 \times 10^{-40}$ \\
				10 & $1.072 \times 10^{20}$ & $1 \times 10^{20}$ & $1 \times 10^{-40}$ & $1 \times 10^{-40}$ \\
				100 & $1.007 \times 10^{20}$ & $1 \times 10^{20}$ & $1 \times 10^{-40}$ & $1 \times 10^{-40}$ \\\hline
			\end{tabular}
		\end{center}
		\caption{$LSR1$, 20 iterations}
		\label{tb:LSR1}
	\end{table}

	\section{Upper and lower bounds on the distance between $\l_1$ and $\l_n$}
	We use the same notation as in the previous section. Notably, $p(x)$ is a monic polynomial given by
	\begin{equation*}
		p(x) = a_nx^n+a_{n-1}x^{n-1}+ \dots +a_0, \text{~~with~~} a_0 \neq 0 \text{~~and~~} n \geq 2,
	\end{equation*}
	and $q(x)$ is its monic reversal polynomial. Let $\l_1, \l_2, \dots, \l_n$ be the roots of $p(x)$ not necessary distinct and listed in non-decreasing order of their absolute values, i.e.,
	$$|\l_1| \leq |\l_2| \leq \dots \leq |\l_n|.$$
	We let $C_1$ and $C_2$ be companion matrices associated, respectively, with $p(x)$ and $q(x)$. We let $N_1$ and $N_2$ be any two matrix norms (not necessarily distinct) on the set of $n\times n$ complex matrices.

	Two quantities that are useful in the numerical analysis of polynomials are the root spread and the absolute root spread of $p(x)$.  The first is the largest distance between two roots of $p(x)$ and has the form $\max_{i,j} |\l_j - \l_i|$, \cite{ED}.  The second one is also called the modulus root spread of $p(x)$ and consists of the largest difference between two absolute values of the roots of $p(x)$ which is simply $|\l_n| - |\l_1|$. Each of these has an upper bound that is an immediate consequence of Theorem \ref{t1}.
	\begin{theorem}\label{t4}
		For every $(k_1, k_2) \in \mathbb{N}^2$, we have
		\begin{equation}
			|\l_n| - |\l_1| \leq \sqrt[k_1]{N_1(C_1^{k_1})} - \sqrt[-k_2]{N_2(C_2^{k_2})},
		\end{equation}
		and
		\begin{equation}
			\max_{i,j} |\l_j - \l_i| \leq 2 \sqrt[k_1]{N_1(C_1^{k_1})}.
		\end{equation}
	\end{theorem}

	The next theorem is about the quantity $|\l_n - \l_1|$ which can be bounded based on Theorem \ref{t1}. It can be regarded as an upper bound on the minimal root distance $\Big(\displaystyle{\min_{i,j}|\l_j - \l_i|}\Big)$, which is useful in analyzing the stability and sensitivity of $p(x)$.
	\begin{theorem} \label{t3}
		For every $(k_1, k_2) \in \mathbb{N}^2$, we have the following.
		\begin{enumerate}
			\item If the points corresponding to $\l_1$ and $\l_n$ in the complex plane lie on the same ray extending from the origin, then
			\begin{equation}
				|\l_n-\l_1| \leq \sqrt[k_1]{N_1(C_1^{k_1})} ~~- \sqrt[-k_2]{N_2(C_2^{k_2})}.
			\end{equation}
			\item If the points corresponding to  $\l_1, \l_n$ and 0 are collinear, but  $\l_1$ and $\l_n$ lie on opposite rays extending from the origin, then for every $\varepsilon > 0$ there exists a positive integer $k_0$ such that
			\begin{equation*}\
				|\l_n-\l_1| > \sqrt[k]{N_1(C_1^{k})} + \sqrt[-k]{N_2(C_2^{k})} - \varepsilon,
			\end{equation*}
			for every integer $k > k_0$.
			\item If the points corresponding to $\l_1, \l_n$ and $0$ are not collinear, then there exist positive numbers $\gamma_1$ and $\gamma_2$ such that for every $\varepsilon_1 \in [0, \gamma_1)$
			and  every $\varepsilon_2 \in [0, \gamma_2)$, there exists an integer $k_0$ such that
			\begin{equation}\label{ff28}
				\sqrt[k]{N_1(C_1^{k})} - \sqrt[-k]{N_2(C_2^{k})} ~+~ \varepsilon_1 < |\l_n - \l_1| < \sqrt[k]{N_1(C_1^{k})} + \sqrt[-k]{N_2(C_2^{k})} ~-~ \varepsilon_2,
			\end{equation}
			for every integer $k > k_0$.
			\item The above three assertions remain valid if $C_1$ is any nonsingular $n \times n$ complex matrix with $n\geq2$ and $C_2$ is its inverse matrix or any other square complex matrix for which the largest and smallest eigenvalues in absolute values are, respectively, $1/\l_1 \text{~ and ~} 1/\l_n$.
			\end{enumerate}
	\end{theorem}
	\begin{proof}
		To relax the notation, we let $U_k = \sqrt[k]{N_1(C_1^{k})}$ and $L_k = \sqrt[-k]{N_2(C_2^{k})}$.  It follows from (\ref{ff24}) that $|\l_n|-|\l_1| \leq U_k - L_k$.  If $\l_1$ and $\l_n$ lie on the same ray extending from the origin, then $|\l_n|-|\l_1| = |\l_n - \l_1|$ and the first assertion follows.

		By (\ref{ff22}), for every positive $\varepsilon$, there exists a positive integer $k_0$ such that $U_k - |\l_n| < \varepsilon$ for every integer $k > k_0$. Combining this with the fact that $|\l_1| \geq L_k$, we obtain $|\l_n| + |\l_1| > U_k + L_k - \varepsilon$. Then the inequality in the second assertion is obtained by using the fact that $|\l_n|+|\l_1| = |\l_n - \l_1|$ when $\l_1, \l_n$ and $0$ are collinear with $0$ lying between $\l_1$ and $\l_n$.

		To prove the third assertion, recall that, in general, $|\l_n|-|\l_1| \leq |\l_n-\l_1| \leq |\l_n| + |\l_1|$ . However, if $\l_1, \l_n$ and $0$ are not collinear in the complex plane, then the inequalities in the above expression become strict, i.e.,
		\begin{equation}\label{ff27}
			|\l_n|-|\l_1| < |\l_n-\l_1| < |\l_n| + |\l_1|.
		\end{equation}
		Note that (\ref{ff27}) is assumed here to be a well-known fact that can be verified by either algebraic or geometric arguments.  Therefore, the numbers $\gamma_1 = |\l_n-\l_1| + |\l_1| - |\l_n|$ and $\gamma_2 = |\l_1| + |\l_n| - |\l_n-\l_1|$ are strictly positive.  Let $a$ be any number in $(0,\gamma_1]$. Then by (\ref{ff22}), there exists a positive integer $k_1$ such that $\big|U_k - |\l_n|\big| < a / 2 ~$ and $~\big||\l_1|-L_k\big| < a / 2~$ for every integer $k>k_1$. Since $U_k \geq |\l_n|$ and $|\l_1| \geq L_k$, it follows that $U_k - |\l_n| < a / 2$ and $ |\l_1|-L_k < a/2$. Consequently we obtain $U_k-L_k - a < |\l_n|-|\l_1| = |\l_n-\l_1| - \gamma_1$ which leads to $U_k-L_k + \gamma_1 - a < |\l_n-\l_1|$. Now set $\varepsilon_1 = \gamma_1 - a$ and notice that since $a\in (0, \gamma_1]$, $\varepsilon_1$ takes all the values in $[0,\gamma_1)$. This is to prove the left inequality in the third assertion.

		To prove the right inequality, let $b$ be any real number in $(0 , \gamma_2]$. There exists a positive integer $k_2$ such that $|\l_1| < L_k + b$ for every integer $k > k_2$. Combine this with the fact that $ |\l_n| \leq U_k$ to get $|\l_n| + |\l_1| < U_k + L_k + b$. Since $|\l_n| + |\l_1| =|\l_n - \l_1| + \gamma_2$, it follows that $|\l_n - \l_1| < U_k + L_k - (\gamma_2 - b)$. Then set $\varepsilon_2 = \gamma_2 - b$ and notice that $\varepsilon_2$ takes all the values in $[0,\gamma_2)$. To complete the proof of the third assertion, take $k_0 = \max \{k_1, k_2\}$.  The fourth assertion follows from Remark \ref{r1}.
	\end{proof}

	\begin{rem}\leavevmode
		\begin{enumerate}
			\item It turns out for some polynomials that $\displaystyle{\max_{ij}|\l_j - \l_i| = |\l_n - \l_1|}$. In this case, Theorem \ref{t3} is providing bounds for the root spread of the polynomial.
			\item The bounds given by Theorem \ref{t4} and Theorem \ref{t3} can be improved using Matrix Balancing in the same way as in the previous section.
		\end{enumerate}
	\end{rem}

	\begin{example}\label{ex4}
		From Figure~2, we can see that the root spread of the polynomial $p(x)$, discussed in Example \ref{ex2}, is $|\l_8 -\l_7| = |\l_8-\l_6| = |7-i|\approx 7.07$.
		We have $$ \max_{ij}{|\l_i-\l_j|}\approx 7.07 \leq 2\sqrt[64]{\|C_1^{64}\|_{\infty}} \approx 8.64,$$ where the value of $\sqrt[64]{N_{\infty}(C_1^{64})}$ is taken from Table~3.  The above inequality is in accordance with Theorem \ref{t4}.  The smallest absolute value is satisfied by four roots: $l_1 = 1, l_2=i, l_3=-i$ and $l_4=-1$.  Neither $\l_2$ nor $\l_3$ lies on the line through $\l_8$ and the origin. Note that $|\l_8 - \l_2| = |\l_8 - \l_3| = \sqrt{17}$ and if we ignore the round-off error, then the third assertion of Theorem \ref{t3} implies
		$$
		\sqrt[64]{\|C_1^{64}\|_{\infty}} - \sqrt[-64]{\|C_2^{64}\|_{\infty}} < \sqrt{17} \approx 4.12 < \sqrt[64]{\|C_1^{64}\|_{\infty}} + \sqrt[-64]{\|C_2^{64}\|_{\infty}}.
		$$
		In fact, from Table~2, we see that
		$$
		\sqrt[64]{\|C_1^{64}\|_{\infty}} - \sqrt[-64]{\|C_2^{64}\|_{\infty}} \approx 3.34,
		\text{~~~and~~~}
		\sqrt[64]{\|C_1^{64}\|_{\infty}} + \sqrt[-64]{\|C_2^{64}\|_{\infty}} \approx 5.3
		$$ \exampleqed
	\end{example}
	%
	\section{Comparison with other bounds and localization regions}
	There are many bounds and corresponding regions for the polynomial roots in the literature.
	We shall briefly discuss some natural connection between our work and some existing results that are based on the study of companion matrices.  Three recent articles that address bounds and locations of polynomial roots are \cite{AM, TDP, V}.  Gershgorin's Theorem is a classical and general way to generate localization regions for the eigenvalues of any square complex matrix.  In \cite{AM}, the author uses this Theorem and leverages the structural properties of the companion matrix to obtain a modification of the original form of the Gershgorin regions associated with a companion matrix. To the best of our knowledge, there are no known techniques that exploit the powers of a given matrix $A$ (whether it is a companion matrix or not) to better improve the Gershgorin region of its spectrum.  In contrast, the bounds obtained in \cite{TDP} and \cite{V} use matrix norms of different companion matrices to bound its spectra. In such cases, the norms of the powers of the companion matrices can be beneficial.  In \cite{TDP}, the authors compare the one, infinity and Frobenius norms of different Fiedler companion matrices, and show that in some situations the bounds obtained by considering some Fiedler companion matrices can be sharper than those obtained from Frobenius companion matrices. Here the plural means a set containing the classical Frobenius matrix and all its permutation similarity transformations including its transpose.  Note that Frobenius companion matrices are contained within the set of Fiedler companion matrices which is contained within the wider set of intercyclic companion matrices.  In \cite{V}, the authors show a set of $n$ Frobenius companion matrices denoted $\left\{ L_b ~~|~~ 0 \leq b \leq n-1 \right\}$ that provides the minimum of one, infinity and Frobenius matrix norms taken over the class of all Fiedler companion matrices, \cite[Thoerems 4.1 and 4.2]{V}.  The question related to our work and which arises now is the following.  \textbf{How do the norms obtained from integer powers of all companion matrices compare to each other?}

	Next we show by examples how this question can be answered for two particular polynomials.
	\begin{example}\label{ex6}
		In this example we look at the particular polynomial $p(x)$ provided in \cite[Example 4.3]{V}, to compare numerically $\sqrt[k]{\|C_1^k\|_{\infty}}$ and $\sqrt[k]{\|L_b^k\|_{\infty}}$.  Let $p(x) = x^8 - 0.1 x^7 -0.1 x^6 - 0.3 x^5 - 0.1 x^4 - 0.5 x^3 - 0.1 x^2 - 0.1 x - 0.1.$ The companion matrix given by (\ref{f17}) is
		\begin{equation}
			C_1 ~~ = ~~ \left[
				\begin{array}{c c c c c c c r }
					0 & 0 &   0 &  0  & 0 &   0 &  0  & 0.1\\
					1 & 0 &   0 &  0  & 0 &   0 &  0  & 0.1\\
					0 & 1 &   0 &  0  & 0 &   0 &  0  & 0.1\\
					0 & 0 &   1 &  0  & 0 &   0 &  0  & 0.5\\
					0 & 0 &   0 &  1 &  0 &   0 &  0  & 0.1\\
					0 & 0 &   0 &  0 &  1 &   0 &  0  & 0.3\\
					0 & 0 &   0 &  0 &  0 &   1 &  0  & 0.1\\
					0 & 0 &   0 &  0 &  0 &   0 &  1  & 0.1
				\end{array}	\right]
		\end{equation}
		According to \cite[Theorem 4.2]{V}, the minimum infinity norm among all Fiedler companion matrices is obtained for the matrix $L_5$ given by
		\begin{equation}
			L_5 ~~ = ~~ \left[
				\begin{array}{c c c c c c c r }
					0   & 1   &   0   &  0   & 0   &   0   &  0  & 0 \\
					0   & 0   &   1   &  0   & 0   &   0   &  0  & 0 \\
					0   & 0   &   0   &  1   & 0   &   0   &  0  & 0 \\
					0   & 0   &   0   &  0   & 1   &   0   &  0  & 0 \\
					0   & 0   &   0   &  0   & 0   &   1   &  0  & 0 \\
					0   & 0   &   0   &  0   & 0   &   0.1 &  1  & 0 \\
					0   & 0   &   0   &  0   & 0   &   0.1 &  0  & 1 \\
					0.1 & 0.1 &   0.1 &  0.5 & 0.1 &   0.3 &  0  & 0
				\end{array}	\right]
		\end{equation}
		Indeed, the infinity norm of $L_5$ is less than or equal to that of  $C_1$:
		$$
		\|L_5\|_{\infty} = 1.2 < \|C_1\|_{\infty} = 1.5.
		$$
		The following table shows upper bounds on the roots of $p(x)$ obtained by the infinity norms of some powers of $C_1$ and $L_5$.
		{\fontsize{8}{11}\selectfont
			\begin{center}
				\begin{tabular}{|c|c|c|c|c|c|c|}\hline
					$k$										& $1$ &  $2$     &  $8$    & $16$    & $32$     & $64$      \\ \hline
					$\sqrt[k]{\| C_1^k\|_{\infty}}\approx$  & $1.5$ 	&  $1.2845$     &  $1.1347$ & $1.1347$ & $1.0949$   & $1.0882$  \\ \hline

					$\sqrt[k]{\| L_5^k\|_{\infty}}\approx$  & $1.2$	&  $1.1446$     &  $1.0918$ & $1.0862$ & $1.0838$   & $1.0826$  \\ \hline
				\end{tabular}	\\
				\vskip 2mm
				Table~9
			\end{center}
		}

		We can see that for every power $k$ in Table~9, the upper bound given by the infinity norm of $L_5$ is better than that given by $C_1$. This is interesting.
		\exampleqed
	\end{example}
	Let's make the comparison for another polynomial.
	\begin{example}
		Consider  $p(x) = x^8 +8x^7+14x^6-28x^5-81x^4-8x^3-14x^2+28x+80$ from Example \ref{ex2}.
		This polynomial does not satisfy the criterion of \cite[Theorem 4.2]{V} which help to determine exactly which one  among the set of matrices $\{L_b\}$ provides the best minimum norm. We again choose the companion matrix $L_5$ which has the same structure as the one used in the previous example but of course with different entries. We obtain the following table in which we are comparing the one, infinity and Frobenius norms obtained from $C_1$ and $L_5$.
		{\fontsize{8}{11}\selectfont
			\begin{center}
				\begin{tabular}{|c|c|c|c|c|c|c|}\hline
					$k$										 & $1$ &  $2$     &  $8$    & $16$    & $32$     & $64$      \\ \hline
					%
					$\sqrt[k]{\| C_1^k\|_{\infty}}\approx$  & $82$ 	&  $26.87$     &  $7.51$ & $5.45$ & $4.67$   & $4.32$  \\ \hline
					$\sqrt[k]{\| C_1^k\|_1}\approx$       	 & $261$	&  $44.12$     &  $8.22$ & $5.70$ & $4.78$   & $4.37$  \\ \hline
					$\sqrt[k]{\| C_1^k\|_F}\approx$       	 & $122.70$ &  $30.97$     &  $7.63$ & $5.49$ & $4.69$   & $4.33$  \\ \hline
					$\sqrt[-k]{\|L_5^k\|_{\infty}} \approx$   & $239$ &  $19.10$     &  $7.03$ & $5.31$ & $4.61$   & $4.29$  \\ \hline
					$\sqrt[-k]{\| L_5^k\|_1} \approx $ 	    & $82$ &  $18.49$     &  $6.53$ & $5.11$ & $4.52$   & $4.25$  \\ \hline
					$\sqrt[-k]{\| L_5^k\|_F} \approx $ 	    & $122.70$ &  $15.97$     &  $6.57$ & $5.12$ & $4.52$   & $4.25$  \\ \hline
				\end{tabular}	\\
				\vskip 2mm
				Table~10
			\end{center}
		}
		\exampleqed
	\end{example}

	From Table~10, we can see again that as $k$ increases, the bounds obtained by using the powers of $L_5$ are better than those obtained by using the powers of $C_1$. Note that the bounds may be improved by the technique of matrix balancing.

	The comparison between the bounds obtained from the powers of $C_1$ and $L_5$ in the previous two examples shows a tendency for relations between the norms to be preserved.  That is, if $C$ and $L$ are two distinct companion matrices for the same polynomial then we might conjecture based on the numerical evidence that if $N(C) \geq N(L)$  then $\sqrt[k]{N(C^k)} \geq \sqrt[k]{N(L^k)}$ for $k > 1$.  If we define the function $b_k$ by
	\begin{equation*}
		b_k(C) = \max\left\{\sqrt[k]{\|C^k\|_1}~,~~ \sqrt[k]{\|C^k\|_{\infty}}~,~~ \sqrt[k]{\|C^k\|_F}\right\},
	\end{equation*}
	we might also conjecture that if $b_1(C) \geq b_1(L)$ then $b_k(C) \geq b_k(L)$ for all $k > 1$.  However, neither of these inequalities hold in general as the following examples show.

\begin{example}
	Consider
	\begin{equation*}
		C =
		\begin{bmatrix}
			0 & 0 & -3/4 \\
			1 & 0 & 1/2 \\
			0 & 1 & -1/2
		\end{bmatrix},\qquad\mbox{and}\qquad
		L = C^\T =
		\begin{bmatrix}
			0 & 1 & 0 \\
			0 & 0 & 1 \\
			-3/4 & 1/2 & -1/2
		\end{bmatrix}.
	\end{equation*}
	It is easily verified that
	\begin{equation*}
		1.75 = \|C\|_1 > \|L\|_1 = 1.5,\qquad\mbox{and}\qquad
		1.458 \approx \sqrt[2]{\|C^2\|_1} < \sqrt[2]{\|L^2\|_1} = 1.5.
	\end{equation*}
\end{example}

\begin{example}
	Consider
	\begin{equation*}
		C =
		\begin{bmatrix}
			0 & 0 & -1/4 \\
			1 & 0 & 1/2 \\
			0 & 1 & 3/4
		\end{bmatrix},\qquad\mbox{and}\qquad
		L =
		\begin{bmatrix}
			0 & 1 & 0 \\
			0 & 3/4 & 1 \\
			-1/4 & 1/2 & 0
		\end{bmatrix}.
	\end{equation*}
	These matrices are distinct companion matrices for the polynomial
	\begin{equation*}
	p(x) = x^3 -\frac{3}{4}x^2 - \frac{1}{2}x + \frac{1}{4}.
	\end{equation*}
	We have
	\begin{equation*}
		\|C\|_1 = 1.5, \qquad
		\|C\|_\infty = 1.75, \qquad
		\|C\|_F\approx 1.696, \qquad
	\end{equation*}
	\begin{equation*}
		\|L\|_1 = 2.25, \qquad
		\|L\|_\infty = 1.75, \qquad
		\|L\|_F\approx 1.696.
	\end{equation*}
	\begin{equation*}
		\sqrt[2]{\|C^2\|_1} \approx 1.225 , \qquad
		\sqrt[2]{\|C^2\|_\infty} \approx 1.677, \qquad
		\sqrt[2]{\|C^2\|_F}\approx 1.322,
	\end{equation*}
	\begin{equation*}
		\sqrt[2]{\|L^2\|_1} = 1.5 , \qquad
		\sqrt[2]{\|L^2\|_\infty} \approx 1.436, \qquad\mbox{and}\qquad
		\sqrt[2]{\|L^2\|_F}\approx 1.376.
	\end{equation*}
	For these matrices we have
	\begin{equation*}
		1.75 = b_1(C) < b_1(L) = 2.25, \qquad\mbox{and}\qquad
		1.677 \approx b_2(C) > b_2(L) = 1.5.
	\end{equation*}
\end{example}

	\section{Conclusion}

	Given a monic complex polynomial $p(x)$, we have seen in this paper how the powers of any of the companion matrices associated with $p(x)$ and its monic reversal polynomial can be used to create an annulus where the roots of $p(x)$ are located. Gelfand's formula ensures that the boundaries of this annulus can be made arbitrarily close to the largest and smallest roots of $p(x)$ in absolute value.  The use of balancing improves the practicality of this approach by delivering greater benefits for smaller powers.  Several ideas and questions arise from this work such as the questions posed in the previous section. The reader may also wonder about the convergence rate of the obtained bounds in terms of the powers of the companion matrices and how it is affected in the case the polynomial has multiple roots. All of these ideas and others are the subject of further research.
	

\begin{thebibliography}{30}
		\bibitem{AE} O. Aberth, Iteration methods for finding all zeros of a polynomial simultaneously, Mathematics of Computations, 27(1973), pp. 339--344.
		
		\bibitem{BF} D.A. Bini, G. Fiorentino, Design, analysis, and implementation of a multiprecision polynomial rootfinder. Numerical Algorithms 23(2000), pp. 127–173.
		
		\bibitem{chgx:07} S. Chandrasekaran, M. Gu, J. Xia, and J. Zhu, A Fast $QR$ Algorithm for Companion Matrices, Operator Theory: Advances and Applications, 179(2007), pp. 111-143.
		
		\bibitem{ED} E. Deutsch, On the Spread of Matrices and Polynomials, Linear Algebra Appl. 22(1978) 49--55.
		
		\bibitem{edmu:95} A. Edelman and H. Murakami, Polynomial roots from companion matrix eigenvalues, Mathematics of Computation, 64(1995), pp. 763--776.
		
		\bibitem{MF} M. Fiedler, A note on companion matrices, Linear Algebra Appl. 372 (2003) 325--331.
		
		\bibitem{CG} C. Garnett, B.L. Shader, C.L. Shader, P. van den Driessche, Characterization of a family of generalized companion matrices, Linear Algebra Appl. 498 (2016) 360–365.
		
		\bibitem{JG} J. Grad, Matrix Balancing, The Computer Journal, 14(1971) 280--284.
		
		\bibitem{H}	P. Henrici, Applied and computational complex analysis, vol 1, John Wiley and Sons, 1974.
		
		\bibitem{HJ} R.A. Horn, C.R. Johnson, Matrix Analysis, 2nd ed, Cambridge University Press, Cambridge, 2013.
		
		\bibitem{Pan2} R. Imbach, V.Y. Pan, Root radii and subdivision for polynomial root-finding, Lecture Notes in Computer Science, Springer, 12865(2021) 136--156.
		
		\bibitem{AM} A. Melman, Modified Gershgorin Disks for Companion Matrices, SIAM Review, 54(2012) 355--373.
		
		\bibitem{Pan1} V.Y. Pan, L. Zhao, Polynomial real root isolation by means of root radii approximation, Lecture Notes in Computer Science, Springer, 9301(2015) 349--360.
		
		\bibitem{TDP} F. De Terán, F. M. Dopico, J. Pérez, New bounds for roots of polynomials based on Fiedler companion matrices, Linear Algebra Appl, 51(2014), pp. 197--230.
		
		\bibitem{vdde:83} P. Van Dooren and P. Dewilde, The eigenstructure of an arbitrary polynomial matrix: Computational aspects, Linear Algebra Appl, 50(1983), pp. 545--579.
		
		\bibitem{V} K. N. Vander Meulen, T. Vanderwoerd, Bounds on polynomial roots using intercyclic companion matrices, Linear Algebra Appl, 39(2018), pp. 94--116.
		
		\bibitem{wilk:63} J. H. Wilkinson, {\em Rounding Errors in Algebraic Processes,} Prentice-Hall, Englewood Cliffs, NJ, 1963.
	\end{thebibliography}
\end{document}